\documentclass{amsart}
\usepackage{amsmath}
\usepackage{amssymb}
\usepackage{graphicx}
\input xy
\xyoption{all}
\usepackage{color}\usepackage{enumerate}
\newtheorem{theorem}{Theorem}[section]

\newtheorem{conjecture}{Conjecture}
\newtheorem{example}{Example}
\newtheorem{lemma}[theorem]{Lemma}
\newtheorem{corollary}[theorem]{Corollary}

\theoremstyle{definition}

\newenvironment{Proof}{{\textit{Proof}.}\ }{~$\square$\vspace{0.2truecm}}

\newcommand{\Char}{\mbox{\rm char}}

\newcommand{\supp}{\mbox{\rm supp}}

\newcommand{\Z}{\mathbb{Z}}

\newcommand{\GL}{{\rm GL}}
\newcommand{\SL}{{\rm SL}}

\newcommand{\ts}[1]{\langle #1\rangle}
\newcommand\tab[1][1cm]{\hspace*{#1}}
\begin{document}

	\title[Permutable subgroups]{Permutable subgroups in $\GL_n(D)$ and applications to locally finite group algebras}

\author[Le Qui Danh]{Le Qui Danh}
\address{Faculty of Mathematics and Computer Science, VNUHCM-University of Science,	227 Nguyen Van Cu Str., Dist. 5, HCM-City, Vietnam; and Department of Mathematics, Mechanics and Informatics, University of Architecture, 196 Pasteur Str., Dist. 3, HCM-City,
	Vietnam}
\email{danh.lequi@uah.edu.vn}	

\author[Mai Hoang Bien]{Mai Hoang Bien}

\author[Bui Xuan Hai]{Bui Xuan Hai}
\address{Faculty of Mathematics and Computer Science, VNUHCM-University of Science,	227 Nguyen Van Cu Str., Dist. 5, HCM-City, Vietnam}
\email{mhbien@hcmus.edu.vn; bxhai@hcmus.edu.vn}

	\keywords{permutable subgroup; free subgroup; general linear group; group algebra; unit group.\\
		\protect \indent 2010 {\it Mathematics Subject Classification.} 16K20; 16S34; 20C99.}
	
	\maketitle
	
\begin{abstract} In this paper we study the existence of free non-abelian subgroups in non-central permutable subgroups of general skew linear groups and locally finite group algebras.
\end{abstract}

\section{Introduction} 
 	
 The question of existence of free groups (in this paper ``free group" means ``non-abelian free group") in subgroups of general skew linear groups have been studied since the last decades of 20-th century. The starting point is Tits' result \cite{Pa_Tits_72}, stating that every finitely generated subgroup of the general linear group $\GL_n(F)$ over a field $F$ either contains a free subgroup or is solvable-by-finite. It is natural to ask what happens if a field $F$ is replaced by a non-commutative division ring $D$. In \cite{Pa_Lichtman_77}, Lichtman constructed a division ring whose the multiplicative group contains a finitely generated non solvable-by-finite subgroup which does not contain free subgroups. The following conjecture due to Lichtman.
 \begin{conjecture}\label{conj:1}
 	The multiplicative group of a non-commutatvie division ring contains a  free subgroup.
 \end{conjecture}
 Later, Gon\c calves and  Mandel  \cite{Pa_Go-Ma_84} posed the more general conjecture.
 \begin{conjecture}\label{conj:2} 
 Any non-central subnormal subgroup of the multiplicative group of a non-commutatvie division ring contains a  free subgroup.
 \end{conjecture}  
 
 Up to now, both these conjectures remain unsolved in general. They are solved affirmatively in several cases:
 
 $\bullet$ Conjecture \ref{conj:1} for division rings with uncountable centers by Chiba in \cite{Pa_Chiba_96};
 
 $\bullet$ Conjecture \ref{conj:2} for centrally finite division rings by Gon\c calves in \cite{Pa_Go_84};
 
 $\bullet$ Conjecture \ref{conj:2} for weakly locally finite division rings by B. X. Hai and N. K. Ngoc in \cite{Pa_Hai-Ngoc_13}.

More generally, the problem on the existence of free subgroups in subnormal subgroups of general skew linear groups is studied in \cite{Bo_Shi-Weh_86} and \cite{Pa_Ngoc-Bien-Hai_17}. 
  	
In this paper, we study the problem on the existence of free subgroups in permutable subgroups of  general skew linear groups, and applying obtained results, we shall study this problem  for groups of units of algebras such as locally finite-dimensional algebras over fields including group algebras of locally finite groups, weakly locally finite division rings and matrix rings over weakly locally finite division rings. We refer to the survey work \cite{Pa_GoRi_13} by Gon\c calves and Rio for more information on free subgroups in groups of unites of group algebras. 

The paper is presented as follows: 

Section 2 introduces the notion of permutable subgroups and some their basic properties. Section 3 devotes to the proof of the fact that in the general skew linear group $\GL_n(D)$ over a division ring $D$ of degree $n\geq 2$, every permutable subgroup is normal. So, the problem on the existence of free subgroups in permutable subgroups of  general skew linear groups of degree $n\geq 2$ reduces to that for normal subgroups. Section 4 study the difference between the notions of permutable subgroups and subnormal subgroups. Moreover, we give some examples of division rings and group algebras whose multiplicative groups contain permutable subgroups that are not subnormal. Using the result of Section 3, in Section 5, we prove that in the general linear group $\GL_n(D)$ over a weakly locally finite division ring $D$, every permutable subgroup contains a free subgroup. In the end, we devote Section 6 to study an open problem on the existence of free subgroups of permutable subgroups in locally finite group algebras.

\section{Definitions and preliminaries} 
	
Given a group $G$, and $N$ its subgroup. We recall some definitions from the theory of groups. 

We say that $N$ is \textit{subnormal} in $G$ if there exists a finite sequence of subgroups of length $r+1$
$$N=N_r\triangleleft N_{r-1}\triangleleft \cdots \triangleleft N_1 \triangleleft N_0=G.$$ 
If no such a sequence of lesser length exists, then we say that $N$ is subnormal of \textit{defect} $r$ in $G$. 

We say that $N$ is  \textit{permutable} in $G$ if $NM=MN$ for any subgroup $M$ of $G$. The following lemma is useful.
\begin{lemma}\label{d2.1}
Let $G$ be a group. The following assertions are equivalent for a subgroup $N$ of $G$.
\begin{enumerate}
				\item For every subgroup $M$ of $G$, the product $NM$ is a subgroup of $G$.
				\item For every subgroup $M$ of $G$, the product  $MN$ is a subgroup of $G$.				
				\item For every subgroup $M$ of $G$, the product $NM$ equals to $MN$.
				\item For every cyclic subgroup $\langle x\rangle$ of $G$, the product $N\langle x\rangle $ equals to $\langle x\rangle N$. 
				\item For every elements $a\in N$ and $x\in G$, there exist $a'\in N$ and  an  integer $n$ such that $ax=x^na'$.
				\item For every elements $a\in N$ and $x\in G$, there exist $a'\in N$ and  an integer $n$ such that $xa=a'x^n$.
\end{enumerate}	
\end{lemma}
\begin{Proof}
The conclusions (1) $\Leftrightarrow$ (2) $\Leftrightarrow$ (3) and (4) $\Leftrightarrow$ (5) $\Leftrightarrow$ (6) are clear by definition. The implication (3) $\Rightarrow$ (4) is trivial. Now we show the implication (6) $\Leftrightarrow$ (1). Assume that (6) holds, that is, For every elements $a\in N$ and $x\in G$, there exist $a'\in N$ and  an integer $n(a,x)$ such that $xa=a'x^{n(a,x)}$. Let $M$ be a subgroup of $G$. We must show that $N M$ is a subgroup. It suffices to show that $ab\in N\cdot M$ for every $a\in M, b\in N$. Indeed, for every $a\in M$ and $b\in N$, there exist $b'\in N$ and an integer $n(a,b)$ such that $ab=b' a^{n(a,b)}$. Hence, $ab\in NM$.
\end{Proof}
		
Permutable subgroups sometimes are called \textit{quasinormal subgroups} (e.g., see \cite{Pa_Gr_75, PaTh_67}). In this paper, we prefer to use the terminology of permutable subgroups according to \cite{Pa_St_72}. For the convenience, we use the symbol $N \le_p G$ to denote the fact that $N$ is a permutable subgroup of $G$. Clearly, every normal subgroup is permutable. However, a subnormal subgroup may not be permutable in a group (a simple example of a dihedral group of order $8$ is given in \cite[Section 13.2]{Bo_Robinson_91}). Later, Stonehever \cite[Theorem B]{Pa_St_72} proves more general result stating that in a finitely generated group, every permutable subgroup is subnormal. So, it is natural to ask whether in any group, every permutable subgroup is subnormal? However, according to Stonehewer, in general, permutable subgroups may not be subnormal (see \cite[Page 2]{Pa_St_72}). For the convenience of further use, we restate this fact in the following lemma.
\begin{lemma}{\rm \cite[Page 2]{Pa_St_72}}\label{non-subnormal} There exists a group containing a permutable subgroup that is not subnormal.
\end{lemma}

Assume that $f : G_1\to G_2$ is a group morphism, $N_1\le _p G_1$ and $N_2\le _p G_2$. Then, by Lemma~\ref{d2.1}, $f(N_1)\le _pf(G_1)$ and $f^{-1}(N_2)\le_p f^{-1}(G_2)$. Moreover, if $N_2$ is non-subnormal permutable in $G_2$, so is $f^{-1}(N_2)$ in $f^{-1}(G_2)$.

Given a group $G$ and $H$ its subgroup. Recall that  $$H_G=Cor_G(H):=\bigcap_{g\in G} g^{-1}Hg$$ is the \textit{core} of $H$ in $G$, and it is clear that $H_G$ is the biggest normal subgroup of $G$ contained in $H$. Also,
$$H^G:=\ts{g^{-1}Hg\vert g\in G}$$
is the \textit{normal closure} of $H$ in $G$ and it is the smallest normal subgroup of $G$ containing $H$. The following lemma due to Gross.
\begin{lemma}\label{lem:2.3} {\rm (\cite[Theorem 2]{Pa_Gr_75})} Let $G$ be any group and $H$ its permutable subgroup. Assume that there exists an infinite cyclic subgroup $C$ of $G$ such that $H\cap C=1$. Then, $H$ is normal in $H^G$ and $H/H_G$ is abelian. 
\end{lemma}
From this lemma, we deduce the following result which is very useful for the next study.			
\begin{lemma}\label{radical} Let $G$ be any group and $H$ its permutable subgroup. Then, either 
\begin{enumerate}
\item 	$G$ is radical over $H$, that is, for every element $x\in G$, there exists a positive integer $n_x$ such that $x^{n_x}\in H$, or
\item $H$ is normal in $H^G$ and $H/H_G$ is abelian group. Consequently, $H$ is subnormal in $G$ of defect at most $2$.
\end{enumerate}

\end{lemma}
\begin{Proof}
Assume that $G$ is not radical over $H$, that is, there exists an element $x\in G$ such that $x^n\not\in H$ for every positive integer $n$. Then, by Lemma \ref{lem:2.3}, $H\trianglelefteq H^G \trianglelefteq G$, so $H$ is subnormal in $G$ of defect at most $2$.
\end{Proof}
\begin{corollary}\label{cor:2.5} Let $D$ be a division ring with center $F$ and $G$ a permutable subgroup of the multiplicative group $D^*$ of $D$. If $G$ is radical over $F$, then $G$ is subnormal in $D^*$ of defect at most $2$, and $G/G_{D^*}$ is an abelian group. 
\end{corollary}
\begin{proof} Clearly we can suppose that $D$ is non-commutative. If $D^*$ is radical over $G$, then $D^*$ is also radical over $F$ because $G$ is radical over $F$ by the supposition. By Kaplansky's theorem in \cite{Pa_Kaplansky_1951}, $D$ is commutative, a contradiction. Hence, $D^*$ is not radical over $G$. By Lemma \ref{radical}, $G$ is subnormal in $D^*$ of defect at most $2$, and $G/G_{D^*}$ is an abelian group.
\end{proof} 
\section{Permutable subgroups in $\GL_n(D)$}
	
Let $D$ be a division ring and $n$ a positive integer. In this section, we study permutable subgroups in the general linear group $\GL_n(D)$ of degree $n$ over $D$. From the definitions, it is evident that in any group, every normal subgroup is permutable. Our main goal in this section is to show that if $n\geq 2$, then every permutable subgroup of $\GL_n(D)$ is normal. 
	
We first need the following function from GAP (a system for computational discrete algebra) to consider the case when $n=2$ and the cardinality of $D$ is at most $3$. This function works as follows: for a finite group $G$, the function checks whether $G$ contains a non-normal permutable subgroup or not. If $G$ contains a non-normal permutable subgroup, then the function will be received \textit{false}; otherwise, it will be \textit{true}.
	\bigskip
	
	{\tt 
	\tab gap>  LoadPackage( "sonata" );

	\tab \tab LoadPackage("permut");

\tab CheckAllPermutableSubgroupsAreNormalSubgroups := function(G)

\tab \tab local lstSubgroups, subgroup, lstPermutableSubgroups, 

\tab \tab lstPermutableSubgroupsNOTNormal;
\bigskip

\tab \tab lstSubgroups := Subgroups(G);

\tab \tab lstPermutableSubgroups := [];

\tab \tab for subgroup in lstSubgroups do

\tab \tab \tab if IsPermutable(G,subgroup) then

\tab\tab\tab\tab Add(lstPermutableSubgroups, subgroup);

\tab\tab\tab fi;

\tab \tab od;
\bigskip

\tab \tab lstPermutableSubgroupsNOTNormal := [];

\tab \tab for subgroup in lstPermutableSubgroups do

\tab \tab \tab if IsNormal(G,subgroup)=false then

\tab\tab\tab\tab Add(lstPermutableSubgroupsNOTNormal, subgroup);

\tab\tab\tab fi;

\tab \tab od;
\bigskip

\tab\tab  if Length(lstPermutableSubgroupsNOTNormal) = 0 then

\tab \tab\tab return true;

\tab \tab else

\tab \tab \tab return false;

\tab \tab fi;

\tab end; }
 	\bigskip
 	
 	Obviously, when the cardinality of $G$ is too large, this function would not work or it would waste too much time. However, we apply this function to just two groups $\GL_2(\Z/2\Z)$ whose cardinality is $6$ and $\GL_2(\Z/3\Z)$ whose cardinality is $48$. In fact, using the above function, we show the following result.
	\begin{lemma}\label{lem3.1} Let $G$ be either $\GL_2(\Z/2\Z)$ or $\GL_2(\Z/3\Z)$.
		Every permutable subgroup of $G$ is normal. $\square$
	\end{lemma}

Now, we consider the cases when either $n>2$ or $n=2$ but $D$ contains at least $4$ elements.
For a division ring $D$, an integer $n>1$ and $1\le i,j\le n$, we denote by $E_{ij}$ the matrix unit, that is, the matrix whose  $(i,j)$-entry is $1$ and all the others are $0$. For each $1\le i\neq j\le n$ and $\alpha\in D$, denote $t_{ij}(\alpha):=I_n+\alpha E_{ij}$,  and $d_i(\alpha):=I_n+(\alpha-1)E_{ii}$. It is clear that $t_{ij}(\alpha)^{-1}=t_{ij}(-\alpha)$, $d_{i}(-1)^{-1}=d_{i}(-1)$ and $d_{i}(-1)t_{ij}(\alpha)d_{i}(-1)=t_{ij}(-\alpha)$. It is well-known that the special linear group $\SL_n(D)$ is generated by the set $\{t_{ij}(\alpha) \mid 1\le i,j\le n, i\ne j, \alpha\in D \}$. 
\begin{theorem}\label{permutable}
Let $D$ be a division ring and $n$ an integer. Assume that either $n>2$ or $n=2$ but $D$ contains at least $4$ elements. Then, every permutable subgroup of $\GL_n(D)$ is normal.
\end{theorem}	 
\begin{Proof}
Let $N$ be a permutable subgroup of $\GL_n(D)$. Clearly, we can assume that $N$ is non-central, and then it suffices to prove that $N$ contains the group $\SL_n(D)$. For $1\le i\neq j\le n$  and $\alpha\in D$, denote by $t'_{ij}(\alpha)=d_{i}(-1)t_{ij}(\alpha)$. We claim that $\SL_n(D)$ is contained in the subgroup of $\GL_n(D)$ generated by $\{t'_{ij}(\alpha)\mid 1\le i\neq j\le n,  \alpha\in D\}$. Observe that  $$t_{ij}(\alpha)=d_{i}(-1)^{-1}t'_{ij}(\alpha)=d_{i}(-1) t'_{ij}(\alpha)=d_{i}(-1)t_{ij}(0) t'_{ij}(\alpha)=t'_{ij}(0) t'_{ij}(\alpha),$$ so the subgroup of $\GL_n(D)$ generated by $\{t_{ij}(\alpha)\mid 1\le i\ne j\le n,  \alpha\in D\}$ is contained in the subgroup generated by $ \{t'_{ij}(\alpha)\mid 1\le i\ne j\le n, \alpha\in D\}$. Since $\{t_{ij}(\alpha)\mid 1\le i\ne j\le n,  \alpha\in D\}$ generates $\SL_n(D)$,   $\SL_n(D)$ is contained in the subgroup generated by $\{t'_{ij}(\alpha)\mid 1\le i\ne j\le n, \alpha\in D\}$. The claim is shown. Now, for  $1\le i\neq j\le n$ and $\alpha\in D$, if $Q=N \langle t'_{ij}(\alpha)\rangle$, then $Q$ is a subgroup of $\GL_n(D)$ since $N$ is permutable in $\GL_n(D)$. As $$t'_{ij}(\alpha)^2=(d_{i}(-1)t_{ij}(\alpha))^2=d_{i}(-1)t_{ij}(\alpha)d_{i}(-1)t_{ij}(\alpha)=t_{ij}(-\alpha)t_{ij}(\alpha)=I_n,$$ the index $[Q:N]$ equals either to $1$ or $2$, which implies that the subgroup $N$ is normal in $Q$. As a corollary, $t'_{ij}(\alpha)N t'_{ij}(\alpha)^{-1}\le N$. Therefore, $N$ is normal in $$\langle t'_{ij}(\alpha)\mid 1\le i,j\le n, i\ne j, \alpha\in D\rangle.$$
In particular, $xNx^{-1}\le N$ for every $x\in \SL_n(D)$ by the above claim. Thus, according to \cite[Theorem 4.9]{Bo_Ar_57} $\SL_n(D)\le N$. The proof is complete.
\end{Proof}
	
Combining Lemma~\ref{lem3.1} and Theorem~\ref{permutable}, we get the main result of this section.
\begin{theorem}\label{main section 3}
Every permutable subgroup of  general skew linear groups of degree $n\geq 2$ is normal. $\square$
\end{theorem}
Before the end of this section, it would be useful to observe that Theorem \ref{permutable} together with \cite[Theorem 3.3]{Pa_Ngoc-Bien-Hai_17} gives the following result.
\begin{theorem}\label{th:3.4} Let $\GL_n(D)$ be the general linear group of degree $n\geq 2$ over an infinite division ring  $D$. Assume that $N$ is a non-central subgroup of  $\GL_n(D)$. Then, the following conditions are equivalent:
	\begin{enumerate}
		\item $N$ is permutable in $\GL_n(D)$.
		\item $N$ is almost subnormal in $\GL_n(D)$.
		\item $N$ is subnormal in $\GL_n(D)$.
		\item $N$ is normal in $\GL_n(D)$.
		\item $N$ contains $\SL_n(D)$.
	\end{enumerate}
\end{theorem}	
\section{The existence of non-subnormal permutable subgroups in Mal'cev-Neumann division rings and group algebras}
As we have seen in Theorem \ref{main section 3}, if $n\geq 2$, then every permutable subgroup of $\GL_n(D)$ is normal. Also, if $D$ is infinite, then Theorem \ref{th:3.4} shows that for $n\geq 2$, in $\GL_n(D)$ a subgroup $N$ is permutable iff it is subnormal. In this section, we are interested in the case $D^*=\GL_1(D)$.  More exactly, we will construct examples of division rings and group algebras whose unit groups contain  permutable subgroups that are not subnormal. Before giving some examples, we present in brief the construction of  Mal'cev-Neumann division rings. 
	 
Let $G$ be a \textit{total ordered group} with order $\preceq$, that is, $\preceq$ is a total order and for every $a,b,c\in G$, if $a\preceq b$, then $ac\preceq bc$ and $ca\preceq cb$. A subset $S$ of $G$ is called \textit{well-ordered} (briefly, \textit{WO}) if every non-empty subset of $S$ has a least element. For $S$ is a non-empty WO subset of $G$, we denoted by $\min(S)$ the least element in $S$. Let $K$ be a field. 
For formal sums of the form $\alpha=\sum\limits_{g\in G}a_gg$, where $a_g\in K$, we set $\supp(\alpha)=\{g\in G\mid a_g\ne 0\}$ and call this subset the {\it support} of $\alpha$. Put $$K((G))=\left\{\alpha=\sum\limits_{g\in G}a_gg\mid \supp(\alpha) \text{ is WO } \right\}.$$ 	 For every $\alpha=\sum\limits_{g\in G}a_gg, \beta=\sum\limits_{g\in G}b_gg\in K((G))$, we define addition and multiplication as follows:  $$\alpha+\beta=\sum\limits_{g\in G}(a_g+b_g)g$$ and $$\alpha\beta=\sum\limits_{t\in G}\left( \sum\limits_{gh=t} a_gb_h\right)t.$$ The above operations are well-defined \cite{T.Y.Lam}, and moreover, we have the following result.
	
\begin{lemma}{\rm \cite[Theorem  14.21]{T.Y.Lam}}\label{lem 4.2}
$K((G))$ is a division ring.
\end{lemma}
The division ring $K((G))$ is called the \textit{Mal'cev-Neumann division ring} of $G$ over $K$. Hence, the group algebra $KG=\left\{\alpha=\sum\limits_{g\in G}a_gg\mid \supp(\alpha) \text{ is finite } \right\}$ is a subalgebra of $K((G))$. To construct our example, we borrow the following result.
	
\begin{lemma} {\rm (\cite[Lemma 10 (1)]{Pa_AaBi_19})} \label{lem 4.3} Let $D=K((G))$ be the Mal'cev-Neumann division ring of a total ordered group $G$ over a field $K$. Then, the function $v : D^*\to G, \alpha\mapsto \min(\supp(\alpha)),$ is a surjective group morphism. $\square$
\end{lemma}

\begin{example}\label{x4.1}{\rm 
Let $H$ be a group such that $H$ contains a subgroup $M$ that is permutable and not subnormal in $H$ as in Lemma~\ref{non-subnormal}. Assume that $H$ is generated by $\{a_i\}_{i\in I}$. Let $G$ be the free group generated by $\{x_i\}_{i\in I}$ with the Magnus order (as in the
proof of \cite[Theorem 6.31]{T.Y.Lam}) and the group morphism $\rho :G\to H$ defined by $\rho(x_i)=a_i$ for every $i\in I$. Then, $\rho$ is surjective. Let $K$ be a field and $D=K((G))$ the Mal'cev-Neumann division ring of $G$ over $K$. Put $$u=\rho \circ v : D^*\to G\to H$$ and $N=u^{-1}(M)=\{\alpha\in D\mid\min\supp(\alpha)\in M \}$. We claim that $N$ is permutable but not subnormal in $D^*$. Since inverse image of a permutable subgroup is permutable, $N$ is permutable in $D^*$. If $N$ is subnormal in $D^*$, that is, there exits a sequence of normal subgroups $$N=N_r\triangleleft N_{r-1}\triangleleft \cdots \triangleleft N_1 \triangleleft N_0=D^*,$$ then $$M=u(N)=u(N_r)\triangleleft u(N_{r-1})\triangleleft \cdots \triangleleft u(N_1) \triangleleft u(N_0)=u(D^*).$$ Since $u$ is surjective, $u(D^*)=H$. Therefore, $M$ is subnormal in $H$, which contradicts the hypothesis. Thus, $N$ is not subnormal in $D^*$.}
\end{example}
\begin{example}\label{group-algebra}{\rm Let $G$, $H$, $M$ and $D=K((G))$ be as in Example~\ref{x4.1}. As we have mentioned, the group algebra $KG$ is a subalgebra of $D$. Moreover, since by the structure of unit elements in $D$, the unit group $(KG)^*= K^* G$.
Now, if $\rho :G\to H$ defined by $\rho(x_i)=a_i$ for every $i\in I$, then $\rho$ is a surjective group morphism, so is $\overline \rho : (KG)^*=K^* G\to H, ag\mapsto \rho(g)$. Hence, $N=(\overline \rho)^{-1} (M)$ is a non-subnormal permutable subgroup of $(KG)^*$.}
\end{example}	

\section{The existence of free subgroups \\in permutable subgroups of $\GL_n(D)$}

The existence of free groups in subnormal subgroups of the general linear group $\GL_n(D)$ over a division ring $D$ have been studied in \cite{Pa_Ngoc-Bien-Hai_17}.
The aim of this section is to investigate the condition of the existence of free groups in permutable subgroups of $\GL_n(D)$. In view of Theorem \ref{main section 3} and \cite[4.5.1]{Bo_Shi-Weh_86}, this problem reduced to the case $n=1$, that is, to the case of $D^*=\GL_1(D)$. For the convenience of 
readers, we restate the result of  \cite[4.5.1]{Bo_Shi-Weh_86} in the following theorem.
\begin{theorem}\label{th:A} {\rm (\cite[4.5.1]{Bo_Shi-Weh_86})} Let $D$ be a division ring which is not an absolute field. Then, every non-central normal subgroup of the general linear group $\GL_n(D)$ of degree $n\geq 2$ contains a free subgroup.
\end{theorem}

For the convenience of further use, we combine Theorem \ref{main section 3} and Theorem \ref{th:A} to state the following result.
\begin{lemma}\label{lem:4.1} Let $D$ be a division ring which is not an absolute field.  Then, every non-central permutable subgroup of the general linear group $\GL_n(D)$ over $D$ of degree $n\geq 2$ contains a free subgroup.
\end{lemma}

Recall that a field $F$ with a prime subfield $P$ is called \textit{absolute} (or \textit{locally finite}) if every finite subset of F generates over $P$ a finite subfield. It is clear that $F$ is absolute if and only if $P$ is finite and $F$ is algebraic over $P$. The class of  non-absolute fields includes ones of characteristic $0$ and uncountable ones (\cite[Corollary B-2.41, p. 343]{Bo_Ro_15}). 

Before stating the next lemma, we recall some definitions. Let $D$ be a division ring with center $F$. We say that $D$ is \textit{centrally finite} if $D$ is a finite vector space over $F$. If $S$ is a subset of $D$, then $F(S)$ denotes the division subring of $D$ generated by the set $S\cup F$. If for every finite subset $S$ of $D$, the division subring $F(S)$ is a finite dimensional vector space over $F$, then $D$ is called \textit{locally finite}. A division ring $D$ is \textit{weakly locally finite} if for every finite subset $S$ of $D$, the division subring $F(S)$ is centrally finite. It was shown in \cite{Pa_Hai-Ngoc_13} that every locally finite division ring is weakly locally finite. In \cite{Pa_Deo-Bien-Hai_19}, the authors have constructed infinitely many weakly locally finite division rings that are not locally finite. Hence, the class of weakly locally finite division rings is large. The following lemma shows that in such division rings, every non-central permutable subgroup contains a free group.

\begin{lemma}\label{lem:4.2} Let $D$ be a non-commutative weakly locally finite division ring and $N$ a non-central permutable subgroup of $D^*$. Then, $N$ contains a free group.
\end{lemma}
\begin{proof} In view of Lemma \ref{radical}, either $N$ is subnormal in $D^*$ or $D^*$ is radical over $N$. If $N$ is subnormal in $D^*$, then by \cite[Theorem 11]{Pa_Hai-Ngoc_13}, $N$ contains a free subgroup. Now, assume that $D^*$ is radical over $N$. By \cite[Theorem 11]{Pa_Hai-Ngoc_13}, there exists a free subgroup in $D^*$ generated by two elements, say, $a, b\in D^*$. Since $D^*$ is radical over $N, a^m\in N$ and $b^r\in N$ for some positive integers $m$ and $r$. Then, the subgroup $\ts{a^m, b^r}$ of $N$ generated by $a^m$ and $b^r$ is free group. 
\end{proof}	
Combining two lemmas above, we get the following theorem.
\begin{theorem}\label{th:4.3} Let $D$ be a division ring, $n$ a positive integer, and $\GL_n(D)$ the general linear group of degree $n$ over $D$. Assume that $N$ is a non-central permutable subgroup of $\GL_n(D)$. Then, $N$ contains a free subgroup if one of the following two conditions holds:
	\begin{enumerate}	
		\item $n\geq 2$ and $D$ is not an absolute field.
		
		\item $n=1$ and $D$ is a non-commutative weakly locally finite division ring. $\square$
	\end{enumerate}	
	\end{theorem}
		
\section{The existence of free subgroups in permutable subgroups of $(FG)^*$}
In this section, we apply the results in the previous sections for an open problem on the existence of free subgroups of permutable subgroups in locally finite group algebras. The existence of free groups in group algebras $FG$ which is inspired by the theory of representations of groups is an interesting topic. For a good survey of the existence of free groups in group algebras, we refer to \cite{Pa_GoRi_13}. Moreover, some negative forms of the existence of free groups have been even received more attention. For example, group algebras whose unit groups satisfy some group identity (e.g., see \cite{Pa_GiJeVa_94,Pa_GiSeVa_97,Pa_GoMa_91,Pa_Li_99, Pa_Me_81}), groups algebras with an Engel unit group (e.g., see \cite{Pa_BiRa_19, Pa_Bo_06, Pa_BoKh_92, Pa_Ra_16, Pa_Ri_00}) and group algebras whose unit groups are nilpotent or solvable (e.g., see \cite{Pa_Ba_71, Pa_BaCo_68, Pa_Bo_05, Pa_BoKh_77, Pa_Kh_72, Pa_Ra_16_2}).

The idea of this section is based on the following fact: if $G$ is finite, then $FG$ is a finite-dimensional vector space over $F$.  Via the Wedderburn-Artin theorem,  modulo the Jacobson radical, $FG$ is isomorphic to a direct product of some matrix rings over division rings. Hence,  modulo the normal subgroup $1+J(FG)$, the unit group $(FG)^*$ is isomorphic to a product of some general  skew linear groups. As a corollary, in some cases, the existence of free groups in $(FG)^*$ reduces to the existence of free groups of general skew linear groups we mentioned in previous sections. Therefore, it is promising to get new results by applying results in previous sections to unit groups  of group algebras. Notice that the algebraic properties of group algebras in two cases of characteristic $0$ and positive characteristic are every different. So, we study these cases separately.
	
\begin{theorem}\label{thm 6.1}
Let $G$ be a locally finite group and $F$ a field of characteristic $0$. Then, every non-abelian permutable subgroup of $(FG)^*$ contains a free subgroup.
\end{theorem}
\begin{Proof} Assume that $N$ is a non-abelian permutable subgroup of $(FG)^*$. Let $\alpha=\sum_{g\in G}\alpha_gg, \beta=\sum_{g\in G}\beta_gg\in N$ such that $\alpha\beta\ne \beta\alpha$. Let $$H=\langle \supp(\alpha)\cup \supp(\beta)\rangle$$ be the subgroup of $G$ generated by $\supp(\alpha)\cup \supp(\beta)$. Since $G$ is locally finite, the subgroup $H$ is finite. Now, consider the group subalgebra $FH$ and the subgroup $M=N\cap (FH)^*$. Then, $FH$ is finite-dimensional over $F$ and $M$ is subnormal in $(FH)^*$. Since $\alpha,\beta$ belong to $(FH)\cap N=M$, the group $M$ is non-abelian. Now, by Lemma~\ref{radical},  either $(FH)^*$ is radical over $N$ or $N$ is subnormal in $(FH)^*$.
\bigskip
		
{\it Case 1:  $(FH)^*$ is radical over $N$.}

\bigskip 

Then, since $N$ is non-abelian, so is $H$. By \cite{Pa_Go_85}, $(FH)^*$ contains a free subgroup of rank $2$, namely, $\langle a,b\rangle$. Moreover, there exist positive integers $n_a,n_b$ such that $a^{n_a}, b^{n_b}\in N$. As a corollary, $\langle a^{n_a}, b^{n_b}\rangle$ is a free subgroup of $N$.
\bigskip
		
{\it Case 2:  $N$ is subnormal in $(FH)^*$.}

\bigskip

Since $\Char(F)=0$, the Jacobson radical $J(FH)=0$ by \cite[Lemma 7.4.2, p.~ 292]{Bo_Pa_77}, which implies that $$FH\cong M_{n_1}(D_1)\times M_{n_2}(D_2)\times \dots \times M_{n_k}(D_k)$$ where $n_1,n_2,\dots, n_k$ are positive integers and $D_1,D_2,\dots,D_k$ are division rings. Moreover, as $FH$ is finite dimensional over $F$, so is every $D_i$ over its center $F_i$. Now we have $(FH)^*\cong \GL_{n_1}(D_1)\times \GL_{n_2}(D_2)\times \dots \times \GL_{n_k}(D_k)$. For every $1\le i\le k$, put $$\pi_i : (FH)^*\to \GL_{n_i}(D_i), (a_1,a_2,\dots,a_k)\mapsto a_i,$$ the  $i$-th projection of $(FH)^*$ onto $\GL_{n_i}(D_i)$. Since $M$ is non-abelian, there exists an integer $1\le i\le k$ such that $\pi_i(M)$ is non-abelian (if all $\pi_1(M), \dots, \pi_k(M)$ are abelian, then $M\le \pi_1(M)\times \dots \times\pi_k(M)$ is abelian, a contradiction). Obviously, $\pi_i(M)$ is subnormal in $\GL_n(D)$, so $\pi_i(M)$ contains a free subgroup by Theorem~\ref{th:4.3}. As a corollary, $M$ contains a free subgroup.
\end{Proof}
	
Now, we move to the positive characteristic case. Let $F$ be a field and $R$ an $F$-algebra. If $R$ is algebraic over $F$, then the Jacobson radical $J(R)$ of $R$ is nil \cite[Corollary 4.19]{T.Y.Lam}. Moreover, 
	
\begin{lemma}\label{lem 6.2} Let $R$ be an algebraic algebra over $F$ with Jacobson radical $J(R)$. If $\overline a \in (R/J(R))^*$, then there exists $u\in R^*$ such that $\overline u=\overline a$.
\end{lemma}
\begin{Proof}
See the first part of the proof of \cite[Lemma 5.1]{Pa_Li_00}.
\end{Proof}
	
\begin{lemma}\label{lem 6.3}
Let $F$ be a field of characteristic $p>0$, and assume that $R$ is an $F$-algebra which is algebraic over $F$. Then, $R^*$ contains a free group if and only if so does $R/J(R)$.
\end{lemma}
\begin{Proof}
Consider the natural projection $\rho : R\to R/J(R), a\mapsto \overline a=a+I$.
Assume that $R^*$ contains a free group $\langle a,b\rangle$ of rank $2$. We claim that $\langle \overline a,\overline b\rangle$ is a free group of $(R/J(R))^*$. Assume that there is a relation $$\overline c_i^{n_1}\overline c_2^{n_2}\dots \overline c_t^{n_t}=\overline 1,$$ where $c_1,c_2,\dots,c_t$ are alternate in $\{a, b\}$  and $n_1,n_2,\dots,n_t$ are non-zero integers. Hence, $1-c_i^{n_1} c_2^{n_2}\dots c_t^{n_t}\in J(R)$. By \cite[Corollary 4.19]{T.Y.Lam}, $1-c_i^{n_1} c_2^{n_2}\dots c_t^{n_t}$ is nil, that is, there exists a positive integer $m$ such that $(1-c_i^{n_1} c_2^{n_2}\dots c_t^{n_t})^{m}=0$. Since $p^m\ge m$, one has $(1-c_i^{n_1} c_2^{n_2}\dots c_t^{n_t})^{p^m}=0$, equivalently, $1-(c_i^{n_1} c_2^{n_2}\dots c_t^{n_t})^{p^m}=0$. Therefore, $1=(c_i^{n_1} c_2^{n_2}\dots c_t^{n_t})^{p^m}$ which contradicts the fact that $\langle a,b\rangle$ is a free group.

Conversely, if $\langle \overline a, \overline b\rangle$ is a free subgroup of $R/J(R)$, then by Lemma~\ref{lem 6.2}, there exist $u,v\in R^*$ such that $\overline u=\overline a$ and $\overline v=\overline b$. Hence, $\langle u, v\rangle$ is a free subgroup of $R^*$.
\end{Proof}
	
Let $G$ be a group and $p$ a prime number. We denote by $O_p(G)$ the maximal normal $p$-subgroup of $G$ if $G$ contains a normal $p$-subgroup. Otherwise, put $O_p(G)=~1$. The following lemma gives a relation between $O_p(G)$ and the Jacobson radical of $FG$.   
\begin{lemma}\label{lem 6.4}\cite[Lemma 2.1]{Pa_Sa_98}
If $G$ is a locally finite group and $F$ is a field of characteristic $p>0$, then $G\cap (1+J(FG))=O_p(G)$.
\end{lemma}

\begin{theorem} \label{thm 6.5} Let $F$ be a non-absolute field of characteristic $p>0$ and $G$ a locally finite group. Assume that $N$ is a non-abelian permutable subgroup of $(FG)^*$ containing $G$. Then, 
\begin{enumerate}
\item If $G'$ is not a $p$-group, then $N$ contains a free subgroup.		
\item If $G'$ is a $p$-group, then $(FG)^*$ contains no free subgroups. 
\end{enumerate}
\end{theorem}
	 
\begin{proof} 1. According to Lemma~\ref{radical}, either $N$ is subnormal in $(FG)^*$ or $(FG)^*$ is radical over $N$. We first consider the case when $N$ is subnormal in $(FG)^*$. We will show that, in this case, if $N$ contains no free subgroups, then $G'$ is a $p$-group. Since $N$ is non-abelian, there exist $x=\sum\limits_{g\in G}a_g g$ and $y=\sum\limits_{g\in G}b_g g$ in $N$ such that $xy\ne yx$. Let $z=g_1h_1g_1^{-1}h_1^{-1}\dots g_mh_mg_m^{-1}h_m^{-1}\in G'$. Let $H$ be the subgroup of $G$ generated by the set $$\{g_1,\dots,g_m, h_1,\dots,h_m\}\cup \supp(x)\cup \supp(y).$$ It is easy to see that the group algebra $FH$ is a subalgebra of $FG$ and is finite-dimensional over $F$ because $H$ is finite.	
	 If $M=N\cap (FH)^*$, then obviously $M$ is subnormal in $(FH)^*$. According to the Wedderburn-Artin theorem, $FH/J(FG)\cong M_{n_1}(D_1)\times M_{n_2}(D_2)\times \dots \times M_{n_k}(D_k)$, where $n_1,n_2,\dots,n_k$ are positive integers and $D_1, D_2,\dots,D_k$ are division rings which are finite-dimensional over their centers $F_1, F_2,\dots, F_k$ respectively. Hence, $$(FH/J(FG))^*\cong \GL_{n_1}(D_1)\times \GL_{n_2}(D_2)\times \dots \times \GL_{n_k}(D_k).$$
	 Put $\rho : (FH)^*\to (FH/J(FG))^*, a\mapsto \overline a,$ and  for every $1\le i\le k$, put $$\pi_i : (FH)^*\to \GL_{n_i}(D_i), (a_1,a_2,\dots,a_k)\mapsto a_i,$$ the  $i$-th projection of $(FH/J(FH))^*$ onto $\GL_{n_i}(D_i)$. If there exists $1\le i\le k$ such that $\pi_i\rho(M)$ is non-abelian in $\GL_{n_i}(D_i)$, then $\pi_i\rho(M)$ contains a free subgroup by Theorem~\ref{th:4.3}, so does $M$, which contradicts the assumption. Hence, $\pi_i\rho(M)$ is abelian for every $1\le i\le k$, which implies that $\rho(M)$ is also abelian, that is, $M(1+J(FH))/(1+J(FH))$ is abelian. Therefore, $M'\subseteq 1+J(FH)$. In particular, $z=1+\alpha$, where $\alpha\in J(FH)$. Observe that $J(FH)$ is nil \cite[Corollary 4.19]{T.Y.Lam}, so $\alpha^s=0$ for some positive integer. It implies that ${z}^{p^s} =(1+\alpha)^{p^s}=1+\alpha^{p^s}=1$, so the order of $z$ is $p^n$ for some non-negative integer $n$. This holds when $z$ ranges over $G'$, so $G'$ is a $p$-group. 
	 
Now, consider the second case when $(FG)^*$ is radical over $N$. Assume that $G'$ is not a $p$-group. Then, according to the first case, $(FG)^*$ contains a free subgroup $\langle a, b\rangle$ of rank $2$. There exist two positive integers $\alpha,\beta$ such that $a^\alpha, b^\beta\in N$, which implies that $\langle a^\alpha, b^\beta\rangle$ is a free subgroup of $N$. The proof of (1) is complete.

2. Assume that $G'$ is a $p$-group. By Lemma~\ref{lem 6.3}, it suffices to show that the unit group $(FG/J(FG))^*$ contains no free subgroups. For $g,h\in G$, we have $gh=chg$, where $c=ghg^{-1}h^{-1}\in G'\subseteq O_p(G)$. By Lemma~\ref{lem 6.4}, $c=1+z$ for some $z\in J(FG)$. Hence, $gh-hg=(1+z)hg-hg=zgh\in J(FG)$. As a corollary, $(F G/J(F G))^*$ is abelian, so it contains no free subgroups. The proof of (2) is complete. 
\end{proof}

\end{document}